\documentclass[a4paper,12pt]{article}
\usepackage[margin=1in]{geometry}
\usepackage[T1]{fontenc}
\usepackage[utf8]{inputenc}
\usepackage{amsmath, amssymb, amsthm}
\usepackage{tikz}
\usepackage{hyperref}
\usepackage{todonotes}
\usepackage{tabularx}
\usepackage{pgfplotstable}
\usepackage{pgfplots}
\usepackage{booktabs}
\pgfplotsset{compat=1.18}
\usepackage{enumitem}

\usepackage[bibencoding=utf8,style=numeric,maxbibnames = 99,giveninits = true]{biblatex}

\newtheorem{definition}{Definition}[section]
\newtheorem{proposition}[definition]{Proposition}
\newtheorem{theorem}[definition]{Theorem}
\newtheorem{lemma}[definition]{Lemma}
\newtheorem{corollary}[definition]{Corollary}
\newtheorem{example}[definition]{Example}

\newtheorem{assumption}[definition]{Assumption}

\newtheorem{algo}[definition]{Algorithm}

\newcommand{\R}{\mathbb{R}}
\newcommand{\N}{\mathbb{N}}
\newcommand{\K}{\mathbb{K}}

\newcommand{\dom}{\operatorname{dom}}

\newcommand{\sign}{\operatorname{sign}}
\newcommand{\dist}{\operatorname{dist}}

\DeclareMathOperator*{\Limsup}{Limsup}

\newcommand{\amsmscLink}[1]{\href{http://www.ams.org/mathscinet/msc/msc2020.html?t=#1}{#1}}

\numberwithin{equation}{section}

\DeclareSourcemap{
  \maps[datatype=bibtex]{
    \map{
      \step[fieldsource=fjournal, fieldtarget=journaltitle]
    }
    \map{
      \step[fieldsource=fseries, fieldtarget=series]
    }
  }
}
\AtEveryBibitem{\clearlist{language}}
\AtEveryBibitem{\clearfield{issn}}
\AtEveryBibitem{\clearfield{isbn}}

\AtBeginBibliography{\small}

\title{Convergence of the Safeguarded Augmented Lagrangian Method under the 
Polyak-\L{}ojasiewicz constraint qualification for Constrained Composite Optimization}
\author{Christian Kanzow and Jannis Krüger}

\begin{document}

\maketitle

\begin{abstract}
\noindent
In this work we provide theoretical and practical results of the Safeguarded Augmented Lagrangian Method (SALM) for constrained composite optimization problems whose objective is the sum of a smooth and a nonsmooth function. We obtain global convergence results to an M-stationary point for SALM under the
Polyak-\L{}ojasiewicz constraint qualification (P\L{}CQ). For this result the
boundedness of the Lagrange multipliers is crucial, and this is shown under 
the assumption that the nonsmooth
part of the objective function is locally Lipschitz continuous. A counterexample
shows that one cannot expect to get bounded multipliers without such an
assumption. The performance of the algorithm is evaluated numerically on a set of sparse portfolio optimization problems with two different regularization terms, one
being Lipschitz and the other one being non-Lipschitz. The results are 
significantly better for the Lipschitz sparsity term, whereas the underlying
method generates seemingly unbounded multipliers in many instances when using the 
non-Lipschitz sparsity function.
\end{abstract}


	\amsmscLink{65K05} $\cdot$ 
	\amsmscLink{90C30}

\section{Introduction}

We consider constrained optimization problems of the form
\begin{align}\label{optProb}
 \min_{x\in \R^n}f(x)+\varphi(x) \quad \text{s.t.}\quad g(x)\leq 0, \quad  h(x)=0
 \end{align}
under the following assumptions:

\begin{assumption}\label{assu:optProb}
 \begin{enumerate}[label = (\roman*)]
  \item The functions $f:\R^n\to \R$, $g:\R^n \to \R^m$ and $h:\R^n\to \R^p$ are continuously differentiable.
  \item $\varphi$ is a locally Lipschitz continuous function.
 \end{enumerate}
\end{assumption}

Problems of this kind arise in a wide range of topics such as engineering, machine learning, optimal control or signal processing, with tasks like image restoration \cite{Bian_2015}, low-rank matrix completion \cite{Cai_2010}, compressed sensing \cite{Wright_2009}, clustering \cite{Song_2007}, principal component analysis \cite{Lu_2011} and model predictive control \cite{De_Marchi_2020}.

The theoretical difficulty for solving problems of such composite nature is the non-smooth function in the objective, as this requires an extension of differentiability and stationarity concepts that may lack the nice properties of smooth problems.

In this paper we consider the Safeguarded Augmented Lagrangian Method (SALM) as a suitable method for numerically solving composite constrained problems. Building on the classical Augmented Lagrangian Method by Powell \cite{Powell_1969}, Hestenes \cite{Hestenes_1969}, and Rockafellar \cite{Rockafellar1973}, the safeguarded version introduced for the \textsc{ALGENCAN} solver \cite{Andreani_2008} is a powerful and versatile choice for smooth constrained optimization due to its algorithmic simplicity and extensive global convergence theory, see also \cite{BirginBook}.

As Augmented Lagrangian methods require solving an unconstrained minimization problem in each step, a practically useful extension of SALM to composite problems necessitates the existence of subproblem solvers which can at least find subdifferentially stationary points of unconstrained composite problems satisfying Assumption \ref{assu:optProb}. More recently, suitable results for the Proximal Gradient Method \cite{Kanzow_2022} and Proximal Quasi Newton Methods \cite{De_Marchi_2022,simeon_deMarchi_2026} have been established. This gave the opportunity to develop a global convergence theory for SALM and constrained composite problems in \cite{De_Marchi_2023}, using even a more general problem setting via geometric constraints, and only requiring $\varphi$ to be proper and
lower semicontinuous (lsc). In the latter paper, the authors use AM-regularity as a constraint qualification in order to show subsequential convergence of the SALM sequence to a M-stationary point for the more general problem.

We want to follow a different approach here. As we consider a more specialized setting with explicit equality and inequality constraints, we are looking to use the least strict constraint qualification tailored to this constraint setting that also still ensures a global convergence result of the iterates. For this approach, we mainly follow the considerations of Andreani et al.\ in \cite{Andreani_2025}. In that paper, the authors introduce a constraint qualification (P\L{}CQ) using a Polyak-\L{}ojasiewicz inequality, show that it is in fact equivalent to an Error Bound Condition (EB), and also show that this is the least restrictive constraint qualification for smooth problems that still ensures boundedness of the multiplier sequences, which is in return pivotal in establishing convergence of the SALM iterates to an M-stationary point.

In view of these results cited above, $\varphi$ being locally Lipschitz continuous may seem like a quite restrictive assumption for the regularization term. However, this restriction is necessary for the required multiplier boundedness, as illustrated in Section \ref{sect:counterex}.

The rest of this work is organized as follows: Section~\ref{Sec:Background}
summarizes some basic facts from variational analysis.
In Section~\ref{sect:convAna}, we extend the global convergence theory of SALM under the P\L{}CQ from \cite{Andreani_2025} to our composite optimization setting \eqref{optProb} under 
Assumption~\ref{assu:optProb}. We then provide a counterexample that 
demonstrates the necessity of the local Lipschitz assumption for $\varphi$ 
in Section~\ref{sect:Lip}. We present some numerical results of SALM applied to 
sparse portfolio optimization comparing Lipschitz and non-Lipschitz regularization functions in Section~\ref{sec:numerical}, and conclude the paper with some
final remarks in Section~\ref{Sec:Conclusions}.

Notation: Let $\Omega$ denote the feasible set of the optimization problem \eqref{optProb}. The active set for a feasible point is written as
\begin{align*}
	I(x^*):=\{i\in \{1,\ldots,m\}:g_i(x^*)=0\}.
\end{align*}
The sets of non-negative and positive real numbers are denoted by $\R_+$ and $\R_{++}$, respectively. For $v\in \R^n$ we write
\begin{align*}
 v_+:=(\max\{0,v_1\},\ldots,\max\{0,v_n\})^T
\end{align*}
for the projection of $ v $ onto the nonnegative orthant $\R^n_+$.
The symbol $\|\cdot\|$ always denotes the Euclidean norm, and $B_r(x)$ and $\overline{B}_r(x)$ the corresponding open and closed balls 
with radius $r$ centered at $x$, respectively. Finally, $\dist(x,S)$ denotes the Euclidean distance $\inf_{w\in S}\|x-w\|$ from $x\in \R^n$ to $S\subseteq \R^n$.

\section{Mathematical Background}\label{Sec:Background}

The set of extended real numbers is denoted by 
$\overline{\R}:=\R\cup \{\pm \infty\}$.
We write $x\to^f x^*$ if $x\to x^*$ and $f(x)\to f(x^*)$, also called $f$\emph{-attentive convergence}. For $f:\R^n\to \overline{\R}$ we write $\dom(f):=\{x\in \R^n\mid  f(x)<+\infty\}$ and say that $f$ is proper, if $\dom(f)\neq \emptyset$ and $f(x)>-\infty$ for all $x\in \R^n$. 

For our convergence analysis we recall some definitions and classical results from nonsmooth analysis as presented in the standard sources \cite{Morduchovic2018} and \cite{Rockafellar2009}.

\begin{definition}
Let $f:\R^n\to\overline{\R}$ be proper, lsc and $x\in \dom (f)$. 
\begin{enumerate}[label = (\alph*)]
\item The set
\begin{align*}
\hat{\partial} f(x):=\left\lbrace v\in\R^n\biggm\vert \liminf_{x\to\bar{x}}\frac{f(x)-f(\bar{x})-v^T(x-\bar{x})}{\|x-\bar{x}\|}\geq 0\right\rbrace
\end{align*}
is called the \emph{Fréchet subdifferential} of $ f $ at $ x $.
\item The set
\begin{align*}
\partial f(x):=\lbrace v\in \R^n \mid \exists \{x^k\}, \{v^k\}\subseteq \R^n : x^k \to^f x, v^k\to v, v^k\in \hat{\partial}f(x^k)\, \forall k\in\N \rbrace
\end{align*}
is called the \emph{limiting/Mordukhovich subdifferential} of $ f $ at $ x $.
\end{enumerate}
\end{definition}

By the definitions of the subdifferentials it is apparent that
\begin{equation}\label{eq:inclusionFrechetMord}
\hat{\partial} f(x)\subseteq \partial f(x)
\end{equation}
holds for all $x\in\dom(f)$.

These subdifferentials admit a generalized Fermat rule \cite[Proposition 1.30 (i)]{Morduchovic2018}, which leads to a natural extension of stationarity for an unconstrained non-smooth setting.

\begin{lemma}[Subdifferential Fermat rule]\label{lem:genFermat}
Let $f$ be proper and lsc and let $f$ have a local minimum in $x^*$. Then
\begin{equation}\label{eq:Fstationary}
0\in \hat{\partial} f(x^*)
\end{equation}
and
\begin{equation}\label{eq:Mstationary}
0\in \partial f(x^*).
\end{equation}
\end{lemma}
Points satisfying these conditions are called \emph{F- or S-stationary} \eqref{eq:Fstationary} (S for strong) and \emph{M-stationary} \eqref{eq:Mstationary} respectively. We can relax the stationarity condition by introducing $\varepsilon$-stationarity. Here, we only introduce this terminology for the Mordukhovich subdifferential.

\begin{definition}
Let $f:\R^n\to\overline{\R}$ be proper and lsc. A vector $x^*\in\dom (f)$ is called an $\varepsilon$\emph{-stationary} point, if there exists an $\eta\in \partial f(x^*)$ with $\|\eta\|\leq \varepsilon$.
\end{definition}

In the following we will list the properties of the Mordukhovich subdifferential used in the convergence analysis. The first result is the robustness of the Mordukhovich subdifferential, which serves as a non-smooth substitute for gradient continuity.

\begin{theorem}[Robustness of the Limiting Subdifferential]\label{thm:robustnessMord}
Let $f:\R^n\to\overline{\R}$ be proper, lsc, and $\bar{x}\in\dom(f)$ be given. Then
\begin{align*}
\partial f(\bar{x})=\Limsup_{x\to^f \bar{x}}\partial f(x):=\{v\in\R^n\mid \exists \{x^k\} \to^f \bar{x},\exists \{v^k\}\to v: v^k\in\partial f(x^k) \ \forall k\in\N \}.
\end{align*}
\end{theorem}

The next important result is the sum rule, which essentially allows us to split our composite problem and deal with the differentiable and non-differentiable part separately.

\begin{theorem}[Sum Rule]\label{thm:SumRuleMord}
Let $f,g:\R^n\to\overline{\R}$ be proper, lsc, and let $\bar{x}\in\dom (f)\cap \dom(g)$ be given. Then
\begin{enumerate}[label = (\alph*)]
\item If $f$ is differentiable in $\bar{x}$, then $\hat{\partial}(f+g)(\bar{x})=\nabla f(\bar{x})+\hat{\partial} g(\bar{x})$.
\item If $f$ is differentiable in a neighborhood of $\bar{x}$ and continuously differentiable in $\bar{x}$, then $\partial(f+g)(\bar{x})=\nabla f(\bar{x})+\partial g(\bar{x})$.
\end{enumerate}
\end{theorem}

If we have local Lipschitz continuity around a given point, the Mordukhovich subdifferential has the following useful boundedness property, cf. \cite[Theorem 9.13]{Rockafellar2009}.

\begin{proposition}\label{prop:LipschitzSubdiffBounded}
Let $f:\R^n\to\overline{\R}$ be locally Lipschitz around some point $\bar{x}\in\dom(f)$ with constant $L=L(\bar{x})>0$. Then $\|v\|\leq L$ for all $v\in \partial f(\bar{x})$.
\end{proposition}

\section{The Safeguarded Augmented Lagrangian Method}\label{sect:convAna}

Recall that the \emph{Lagrangian function} of the optimization problem \eqref{optProb} is given by
\begin{align*}
\mathcal{L}(x,\lambda,\mu)=f(x)+\varphi(x)+\lambda^Tg(x)+\mu^Th(x),
\end{align*}
whereas the (Hestenes-Powell-Rockafellar) \emph{augmented Lagrangian function} with penalty parameter $\rho>0$ is defined by
\begin{align*}
\mathcal{L}_\rho(x,\lambda,\mu)=f(x)+\varphi(x)+\frac{\rho}{2}\left(\left\|\left( g(x)+\frac{\lambda}{\rho}\right)_+\right\|^2+\left\|h(x)+\frac{\mu}{\rho}\right\|^2\right)
\end{align*}
for some $x\in\R^n$ and certain vectors $\lambda\in \R^m_+$, $\mu\in\R^p$,
usually called the \emph{Lagrange multipliers}.

We define the Safeguarded Augmented Lagrangian Method (SALM) similar to \cite{Andreani_2025} as follows.

\begin{algo} (Safeguarded Augmented Lagrangian Method)
\label{algo:SALM}
\begin{enumerate}[label=(S.\arabic*),start=0]
\item Let $u_{\max}>0$, $v_{\min}< v_{\max}$,  $\{u_k\}\subseteq [0,u_{\max}]^m$ and $\{v_k\}\subseteq [v_{\min},v_{\max}]^p$ be bounded sequences. Let $\{\varepsilon_k\}\subseteq \R_{++}$ with $\varepsilon_k\to 0$. Choose a starting point $(x^0,\lambda^0,\mu^0)\in \R^n\times\R^m_+\times \R^p$ and a $\rho_0>0$ as well as parameters $\tau\in (0,1),\gamma>1$ and set $k:=0$.
\item If a suitable termination criterion holds for $(x^k,\lambda^k,\mu^k)$: STOP.
\item Compute $x^{k+1}$ as an $\varepsilon_k$-stationary point of
\begin{align*}
\min_{x\in\R^n} \mathcal{L}_{\rho_k}(x,u^k,v^k).
\end{align*}
\item Set
\begin{align}\label{algoMultiplierUpdate}
\begin{split}
\lambda_i^{k+1}&:=(u_i^k+\rho_kg_i(x^{k+1}))_+ \quad \forall i=1,\ldots,m, \\
\mu_j^{k+1}&:=v^k_j+\rho_kh_j(x^{k+1}) \quad \forall j=1,\ldots,p .
\end{split}
\end{align}
\item If either $k=0$ or the two conditions
\begin{align}\label{algoCaseOne}
\begin{split}
\|h(x^{k+1})\|&\leq \tau \|h(x^k)\| \quad \text{and} \\
\left\| \min\left\lbrace -g(x^{k+1}),\frac{u^k}{\rho_k}\right\rbrace \right\| &\leq \tau \left\| \min\left\lbrace -g(x^{k}),\frac{u^{k-1}}{\rho_{k-1}}\right\rbrace \right\|
\end{split}
\end{align}
hold, set $\rho_{k+1}:=\rho_k$, otherwise set $\rho_{k+1}:=\gamma \rho_k$.
\item Set $k\leftarrow k+1$ and go to (S.1)
\end{enumerate}
\end{algo}

The safeguarding mechanism, which is to use the bounded sequences $\{u^k\}, \{v^k\}$ in the augmented Lagrangian for unconstrained subproblems, yields better global convergence properties, as discussed in \cite{Kanzow_2017}.

The multiplier update in \eqref{algoMultiplierUpdate} is done such that the following relation between Lagrangian and augmented Lagrangian holds in each step.

\begin{lemma}\label{lagrangeEquivalence}
Let $k\in \N$. It holds
\begin{align*}
\partial_x \mathcal{L}(x^k,\lambda^{k},\mu^{k})=\partial_x\mathcal{L}_{\rho_{k-1}}(x^k,u^{k-1},v^{k-1})
\end{align*}
for the iterates of the safeguarded augmented Lagrangian method.
\end{lemma}

\begin{proof}
By the sum rule for the Mordukhovich subdifferential and the continuous differentiability of $f,g,h$, we have
\begin{align*}
\partial_x^M\mathcal{L}_{\rho_{k-1}}(x^k,u^{k-1},v^{k-1})
&=\nabla f(x^k)+\partial_x \varphi(x^k) \\
&\quad+\sum_{i=1}^m(u_i^{k-1}+\rho_{k-1} g_i(x^k))_+ \nabla g_i(x^k) \\
&\quad+\sum_{j=1}^p(v_j^{k-1}+\rho_{k-1}h_j(x^k)) \nabla h_j(x^k) \\
&= \nabla f(x^k)+\partial_x \varphi(x^k)+\sum_{i=1}^m \lambda_i^{k} \nabla g_i(x^k)+\sum_{j=1}^p \mu_j^{k} \nabla h_j(x^k)\\
&=\partial_x \mathcal{L}(x^k,\lambda^{k},\mu^{k}),
\end{align*}
where the penultimate identity follows from the definitions of the updated
multipliers $ \lambda_i^k $ and $ \mu_j^k $ in Algorithm~\ref{algo:SALM}.
\end{proof}

In smooth constrained minimization problems, we are aiming for convergence of our algorithm to a point satisfying the KKT conditions. These conditions are necessary for local minima under a suitable constraint qualification, and are also algorithmically verifiable in a reasonable way. For our potentially non-smooth setting, the counterpart of the KKT conditions is the notion of M-stationarity that is applicable to our case. We do this by adapting the generalized definition of composite M-stationarity from \cite{De_Marchi_2023} to our explicit constraint setting, which leads to the following definition.

\begin{definition}
The triple $(x^*,\lambda^*,\mu^*)\in\R^n\times\R^m_+\times\R^p$ is called an \emph{M-stationary point} of problem \eqref{optProb} if the following 
conditions hold:
\begin{align*}
0&\in \nabla f(x^*)+\partial \varphi(x^*)+ \sum_{i=1}^m\lambda^*_i\nabla g_i(x^*)+\sum_{j=1}^p\mu^*_j\nabla h_j(x^*)\quad\text{(stationarity)}, \\
h(x^*)&=0 \quad\text{(equality feasibility)},\\
g(x^*)&\leq 0 \quad\text{(inequality feasibility)}, \\
(\lambda^*)^Tg(x^*)&=0\quad\text{(complementarity)}.
\end{align*}
\end{definition}

In our subsequent convergence analysis, we only consider the case that the SALM Algorithm runs indefinitely and produces an infinite series of iterates. In practice, we want to employ a termination criterion in (S.1) that checks for (approximate) stationarity, for more details we refer to Section \ref{sec:numerical}.

Ensuring the boundedness of the sequences generated by SALM is already sufficient for convergence to an M-stationary point on a subsequence, as stated in the following theorem.

\begin{theorem}\label{thm:bounded->KKT}
Let $ \big\{ (x^k, \lambda^k, \mu^k) \big\}$ be an infinite sequence generated by Algorithm \ref{algo:SALM}. Then every accumulation point of this sequence is an M-stationary
point of the program \eqref{optProb}.
\end{theorem}

\begin{proof}
Let $ (x^*, \lambda, \mu) $ be an accumulation point of the sequence 
$ \big\{ (x^k, \lambda^k, \mu^k) \big\}$, and let 
$ \big\{ (x^k, \lambda^k, \mu^k) \big\}_K $ be a subsequence converging
to this accumulation point. We proceed by showing that the M-stationarity 
conditions hold in $(x^*,\lambda,\mu)$.

To this end, we begin with the feasibility of $ x^* $. Taking the multiplier 
update \eqref{algoMultiplierUpdate} and rearranging, we obtain
\begin{align*}
\rho_{k-1}g_i(x^k)_+\leq (u_i^{k-1}+\rho_{k-1}g_i(x^k))_+= \lambda_i^k
\end{align*}
for $i=1,\ldots,m$
and
\begin{align*}
\rho_{k-1}|h_j(x^k)|&=\sign(h_j(x^k))\rho_{k-1}h_j(x^{k}) \\
&= \sign(h_j(x^k))(\mu^k_j-v_j^{k-1}) \\
& \leq  |\mu_j^k| + \max \big\{ |v_{\min}|, | v_{\max} | \big\}
\end{align*}
for $j=1,\ldots,p$. By the boundedness of the multiplier sequences on $K$, 
the right-hand sides of the respective equations are bounded for $k\in K$, hence there exists a constant $c>0$ such that
\begin{align}\label{boundedByM}
\rho_{k-1}g_i(x^k)_+\leq c, \quad \rho_{k-1}|h_j(x^k)| \leq c
\end{align}
for all $i=1,\ldots,m$, $j=1,\ldots,p$, $k\in K$.

We first assume that the sequence $ \{ \rho_k \} $ is unbounded. Due to the
updating of the penalty parameter, this already implies $ \{ \rho_k \} \to 
\infty $. In particular, we therefore have $\lim_{k\in K}\rho_{k-1}=\infty$. 
Then \eqref{boundedByM} yields
\begin{align*}
\lim_{k\in K}g_i(x^k)_+=0  \quad\text{and}\quad \lim_{k\in K}|h_j(x^k)|=0
\end{align*}
for all $i=1,\ldots,m$ and $j=1,\ldots,p$. Since $\lim_{k\in K}x^k=x^*$ and $g,h$ are continuous, we therefore obtain $g(x^*)\leq 0$ and $h(x^*)=0$, 
which immediately shows feasibility of $x^*$.

On the other hand, if $ \{ \rho_k \} $ is bounded, this sequence is eventually
constant. Thus the condition \eqref{algoCaseOne} cannot be violated infinitely 
many times, and in infinitely many iterations we have a decrease of at least $\tau\in(0,1)$. This immediately yields $\lim_{k\to\infty}\|h(x^k)\|= 0$ and thus $h(x^*)=0$. Similarly, we obtain
\begin{align}\label{VkToZero}
\lim_{k\to\infty}\min\left\lbrace -g(x^k),\frac{u^{k-1}}{\rho_{k-1}}\right\rbrace=0.
\end{align}
Assume there exists an index $i\in\{1,\ldots,m\}$ with $g_i(x^*)>0$. Then it holds
\begin{align*}
\min\left\lbrace -g_i(x^k),\frac{u_i^{k-1}}{\rho_{k-1}}\right\rbrace\leq -g_i(x^k)\leq -\frac{1}{2}g_i(x^*)<0
\end{align*}
for $k$ sufficiently large, which produces a contradiction by taking the limit $k\to\infty$. Thus we must have $g_i(x^*)\leq 0$ for all $=1,\ldots,m$, and feasibility of $x^*$ holds also in this case.

We next verify the stationarity of the limit $ (x^*, \lambda, \mu) $.
In Step 2 of the SALM, we compute $x^{k}$ as an $\varepsilon_k$-stationary point of the augmented Lagrangian, hence there exists an $\eta_k\in \partial_x \mathcal{L}_{\rho_{k-1}}(x^k,u^{k-1},v^{k-1})$ with $\|\eta_k\|\leq \varepsilon_k$.  By Lemma \ref{lagrangeEquivalence} we have $\eta_k\in \partial_x\mathcal{L}(x^k,\lambda^k,\mu^k)$. As all functions except $\varphi$ in the Lagrangian function are continuously differentiable, this is equivalent to the existence of a $z^k\in \partial_x \varphi(x^k)$ with
\begin{align*}
\left\|\nabla f(x^k)+z^k+\sum_{i=1}^m \lambda_i^k\nabla g_i(x^k)+\sum_{j=1}^p \mu_j^k \nabla h_j(x^k)\right\|\leq \varepsilon_k.
\end{align*}
We now define
\begin{equation}\label{eq:defEk}
 E^k:=\nabla f(x^k)+z^k+\sum_{i=1}^m \lambda_i^k\nabla g_i(x^k)+\sum_{j=1}^p \mu_j^k \nabla h_j(x^k)
\end{equation}
and
$z^*:=-\nabla f(x^*)-\sum_{i=1}^m \lambda_i\nabla g_i(x^*)-\sum_{j=1}^p\mu_j\nabla h_j(x^*)$. Then
\begin{align*}
\|z^k-z^*\|&=\left\|z^k+\nabla f(x^k)+\sum_{i=1}^m \lambda_i^k\nabla g_i(x^k)+\sum_{j=1}^p \mu_j^k \nabla h_j(x^k)\right. \\
&\quad -\nabla f(x^k)-\sum_{i=1}^m \lambda_i^k\nabla g_i(x^k)-\sum_{j=1}^p \mu_j^k \nabla h_j(x^k) \\
&\quad \left.+\nabla f(x^*)+\sum_{i=1}^m \lambda_i\nabla g_i(x^*)+\sum_{j=1}^p\mu_j\nabla h_j(x^*)\right\| \\
&\leq  \left\|E^k\right\| +\|\nabla f(x^*)-\nabla f(x^k)\| \\
&\quad +\left\|\sum_{i=1}^m\lambda_i \nabla g_i(x^*)-\sum_{i=1}^m \lambda^k_i \nabla g_i(x^k)\right\|+\left\|\sum_{j=1}^p \mu_j \nabla h_j(x^*)-\sum_{j=1}^p \mu_i^k \nabla h_j(x^k)\right\| \\
&\leq \varepsilon_k+\|\nabla f(x^*)-\nabla f(x^k)\| \\
&\quad +\left\|\sum_{i=1}^m\lambda_i \nabla g_i(x^*)-\sum_{i=1}^m \lambda^k_i \nabla g_i(x^k)\right\|+\left\|\sum_{j=1}^p \mu_j \nabla h_j(x^*)-\sum_{j=1}^p \mu_i^k \nabla h_j(x^k)\right\|
\end{align*}
Taking the limit $k\to\infty$, $k\in K$ and using $\varepsilon_k\searrow 0$ as well as the continuous differentiability of $f,g,h$, the right-hand side vanishes, and we have $\lim_{k\to\infty,k\in K}z^k =z^*$. By the robustness of the Mordukhovich subdifferential, this shows $z^*\in \partial \varphi(x^*)$ and thus
\begin{align*}
0\in \nabla f(x^*)+\partial \varphi(x^*)+ \sum_{i=1}^m\lambda_i\nabla g_i(x^*)+\sum_{j=1}^p\mu_j\nabla h_j(x^*),
\end{align*}
so that the triple $ (x^*, \lambda, \mu) $ satisfies the stationarity condition.

We finally establish satisfaction of the complementarity property.
To this end, consider an arbitrary index $i\in\{1,\ldots,m\}$. If 
$g_i(x^*)=0$, we are already finished. Hence, let $g_i(x^*)<0$. If 
$\lim_{k \to \infty}\rho_{k-1}=\infty$, we obtain by the boundedness of $\{u^k\}$ 
\begin{align*}
\lim_{k\in K}u^{k-1}_i+\rho_{k-1} g_i(x^{k})=-\infty.
\end{align*}
In particular, for $k\in K$ large enough, we have
\begin{align*}
\lambda_i^k=(u^{k-1}_i+\rho_{k-1} g_i(x^{k}))_+=0,
\end{align*}
and thus $\lambda_i=0$ in the limit $k\in K$. We now assume that $\rho_{k}$ is bounded, hence, eventually constant. By the same reasoning that led to \eqref{VkToZero} and $g_i(x^*)<0$, we must have
\begin{align*}
\min\left\lbrace -g_i(x^k),\frac{u_i^{k-1}}{\rho_{k-1}}\right\rbrace = \frac{u_i^{k-1}}{\rho_{k-1}}
\end{align*}
for sufficiently large $k\in K$. As $\{\rho_{k-1}\}$ is eventually constant, 
taking the limit on $ K $ yields $\lim_{k\in K}u_i^{k-1}=0$. For all $k\in K$ sufficiently large, we therefore have
\begin{align*}
u_i^{k-1}+\rho_{k-1}g_i(x^k)<0
\end{align*}
and thus by \eqref{algoMultiplierUpdate}
\begin{align*}
\lambda_i^{k}= (u_i^{k-1}+\rho_{k-1}g_i(x^k))_+=0.
\end{align*}
In the limit $k\in K$, we obtain $\lambda_i=0$. Hence, for all $i=1,\ldots,m$, 
we have shown $\lambda_i g_i(x^*)=0$.
\end{proof}

Note that in the proof we did not have to rely on the Lipschitz-continuity of $\varphi$, hence this result is also valid if $\varphi$ is merely continuous.
However, we need the Lipschitz property of $\varphi$ to establish sufficient conditions for multiplier boundedness, and this is also necessary in the sense that applications of SALM on non-Lipschitz functions can lead to unbounded multipliers, as discussed in Section \ref{sect:counterex}.

We can ensure multiplier boundedness if certain constraint qualifications hold in the limit point $x^*$, which are taken from \cite{Andreani_2025} and defined as follows.

\begin{definition}
We define the infeasibility measure $\Phi:\R^n\to\R_+$ as
\begin{align}\label{eq:infeasMeasure}
\Phi(x):=\frac{1}{2}(\|g(x)_+\|^2+\|h(x)\|^2).
\end{align}
We say that $\Phi$ meets the \emph{Polyak-\L{}ojasiewicz (P\L{}) inequality} at $x^*\in\Omega$ when there are constants $\nu>0$ and $\delta>0$ such that
\begin{align}
\sqrt{\Phi(x)}\leq \nu\|\nabla \Phi(x)\| \quad \forall x\in B_\delta(x^*).
\end{align}
We call $\nu$ a \emph{Polyak-\L{}ojasiewicz constant} at $x^*$. If $\Phi$ satisfies the P\L{} inequality at $x^*$ we say that the \emph{Polyak-\L{}ojasiewicz constraint qualification (P\L{}CQ)} holds at $x^*$.
\end{definition}

This type of inequality was first studied by Polyak in \cite{Polyak_1963}, and was identified as a special instance of the \L{}ojasiewicz inequality \cite{zbMATH03371284}. Different to our application here, it was formulated for an optimality measure $\Theta(x):=f(x)-f^*$, where $f^*$ is the optimal value of an unconstrained minimization problem with objective $f$. This is why, since its introduction, the inequality has been mainly used to prove convergence rate results for unconstrained optimization methods like gradient descent and 
proximal gradient methods, see also \cite{AttouchBolteRedontSoubeyran2010,Karimi_2016}.

In the approach by Andreani et al.\ \cite{Andreani_2025}, 
the P\L{} inequality for the infeasibility measure $\Phi$ as a constraint qualification yields a sufficient condition for multiplier boundedness. Additionally, Andreani et al. show in \cite[Theorem 3.2, Theorem 4.1]{Andreani_2025} that, in the case of the constraint functions $g,h$ being also locally Lipschitz continuous, the P\L{}CQ is equivalent to their own Relaxed Quasi-Normality (RQN) constraint qualification from \cite{Andreani2024} as well as the commonly used error bound constraint qualification (EB). The latter is defined as follows.

\begin{definition}
We say that a feasible $x^*\in \Omega$ satisfies the \emph{Error Bound (EB) constraint qualification} or \emph{Error Bound condition} if there exist constants $\kappa>0$ and $\delta >0$ such that
\begin{align*}
\dist (x,\Omega)\leq \kappa \sqrt{\Phi(x)} \quad \forall x\in B_\delta(x^*)
\end{align*}
where $\kappa$ is called the \emph{local error bound constant} for $\Omega$ at $x^*$ relative to the infeasibility measure $\Phi$.
\end{definition}

The error bound condition was first introduced in \cite{Hoffman1952OnAS} for systems of linear inequalities. Similar to the P\L{}CQ, it has often been used in unconstrained minimization to connect an optimality measure in a point $x$ to the distance of $x$ to a set of minimizers $S$ in order to obtain results about convergence rates (cf. \cite{Drusvyatskiy_2018,Necoara_2018}). However, there are also instances in the literature where the EB condition is applied to an infeasibility measure and the distance to the feasible set, like in \cite{Pang_1997}. In the latter, the EB condition is used to derive an exact penalty result for a constrained problem.

In the following, we show that if the P\L{}CQ holds in an accumulation point of a SALM sequence, we obtain multiplier boundedness. In the case of locally Lipschitz gradients of the constraint functions, this result can also be extended to the equivalent EB and RQN constraint qualifications. In the latter setting, we also note that on the contrary, if any of the three equivalent CQs is violated, SALM can produce unbounded multiplier sequences even for completely smooth objective functions, as shown in \cite[Theorem 3.3]{Andreani_2025}, which is why these CQs can be seen as the minimally restrictive ones for ensuring multiplier boundedness.

\begin{theorem}\label{thm:lagrangeBounded}
Let $x^*$ be feasible and assume that $g,h$ are continously differentiable. Let additionally $\varphi$ be locally Lipschitz continuous in a neighborhood of $x^*$. If the P\L{}CQ holds at $x^*$ and there exists an infinite subset $K\subseteq \N$ such that the sequence $\{x^k\}$ generated by the SALM converges on $K$ to $x^*$, then the associated dual subsequences $\{\lambda^k\}_{k\in K}$ and $\{\mu^k\}_{k\in K}$ are bounded.
\end{theorem}
\begin{proof}
Let $\{x^k\}\subseteq \R^n$ and $\{\rho_k\}\subseteq \R_{++}$ be infinite sequences generated by the safeguarded augmented Lagrangian method. By the definition of the algorithm and similar reasoning as in the proof of Theorem \ref{thm:bounded->KKT}, we have that there exists $z^k\in\partial \varphi(x^k)$ and $\{E^k\}$ as defined in \eqref{eq:defEk} such that
\begin{align}\label{eq:lagrangeBounded:eps-stationarity}
\|E^k\|\leq \varepsilon_k
\end{align}
for all $k\in K$. By $\varepsilon_k\searrow 0$, we obtain $\lim_{k\to\infty}E^k=0$. As $\varphi$ is locally Lipschitz at $x^*$, there exist $L>0$ and $\delta>0$ 
such that
\begin{align*}
|\varphi(x)-\varphi(y)|\leq L\|x-y\|
\end{align*}
for all $x,y\in B_\delta(x^*)$. Proposition \ref{prop:LipschitzSubdiffBounded} now implies $\|z^k\|\leq L$ for all $k\in \K$ sufficiently large. This again implies by the above results and the continuous differentiability of $f$ that
\begin{equation}
\|\nabla f(x^k)+z^k-E^k\|\leq \|\nabla f(x^k)\| + \|z^k\|+\|E^k\|\leq \|\nabla f(x^*)\|+L+1
\end{equation}
for $k\in K$ sufficiently large. We define $M_1:=\max\{u_{\max},|v_{\min}|,|v_{\max}|\}$. By continuous differentiability of $g,h$, there exists a constant $M_2>0$ such that
\begin{align*}
\|\nabla g_i(x^k)\|\leq M_2 \text{ and }\|\nabla h_j(x^k)\|\leq M_2
\end{align*}
for all $i=1,\ldots ,m$ and $j=1,\ldots,p$ as well as $k\in K$ sufficiently large.
Altogether we have
\begin{align*}
\rho_{k-1}\|\nabla \Phi(x^k)\|&=\left\|\sum_{i=1}^m \rho_{k-1}(g_i(x^k))_+\nabla g_i(x^k)+\sum_{j=1}^p \rho_{k-1}h_j(x^k)\nabla h_j(x^k)\right\| \\
&= \left\|\sum_{i=1}^m [\rho_{k-1}(g_i(x^k))_+ -\lambda_i^k]\nabla g_i(x^k) \right.\\
 &\quad \left. +\sum_{j=1}^p [\rho_{k-1}h_j(x^k)-\mu_j^k]\nabla h_j(x^k)+E^k-\nabla f(x^k)-z^k\right\| \\
&= \left\|\sum_{i=1}^m [\rho_{k-1}(g_i(x^k))_+ -(\rho_{k-1}g_i(x^k)+u_{i}^{k-1})_+]\nabla g_i(x^k) \right.\\
 &\quad \left. +\sum_{j=1}^p [\rho_{k-1}h_j(x^k)-(\rho_{k-1}h_j(x^k)+v_j^{k-1})]\nabla h_j(x^k)+E^k-\nabla f(x^k)-z^k\right\| \\
&\leq  \sum_{i=1}^m u_{i}^{k-1}\|\nabla g_i(x^k)\|+\sum_{j=1}^p |v_j^{k-1}|\|\nabla h_j(x^k)\|+\|E^k-\nabla f(x^k)-z^k\| \\
&\leq \sum_{i=1}^m M_1M_2 +\sum_{j=1}^p M_1M_2 + \|\nabla f(x^*)\|+L+1 \\
&= (m+p)M_1M_2 +\|\nabla f(x^*)\|+L+1=:S
\end{align*}
for sufficiently large $k\in K$, where the first inequality follows from 
the nonexpansiveness of the projection operator. By assumption, the 
P\L{}-inequality holds in $x^*$. By combining it with the above inequality,
we obtain
\begin{align*}
\rho_{k-1}\sqrt{\Phi(x^k)} \leq \nu \rho_{k-1}\|\nabla \Phi(x^k)\|\leq \nu S
\end{align*}
for $k \in K $ sufficiently large, where $ \nu > 0 $ denotes the 
P\L\ constant. For all $k\in K$ large enough, we have
\begin{align*}
\|\lambda^k\|&=\|(u^{k-1}+\rho_{k-1}g(x^k))_+\| \\
&\leq \|(u^{k-1}+\rho_{k-1}g(x^k))_+ -\rho_{k-1} g(x^k)_+\| +\rho_{k-1}\|g(x^k)_+\| \\
&\leq \|u^{k-1}\|+\rho_{k-1}\sqrt{2\Phi(x^k)} \\
&\leq \sqrt{m}M_1+\sqrt{2}\nu S
\end{align*}
and, analogously,
\begin{align*}
\|\mu^k\|&=\|v^{k-1}+\rho_{k-1}h(x^k)\| \\
&\leq \|v^{k-1}\| +\rho_{k-1}\|h(x^k)\| \\
&\leq \|v^{k-1}\|+\rho_{k-1}\sqrt{2\Phi(x^k)} \\
&\leq \sqrt{p}M_1+\sqrt{2}\nu S.
\end{align*}
This shows the boundedness of the multiplier sequences $\{\lambda^k\}$ and $\{\mu^k\}$ on $K$.
\end{proof}

Through our convergence analysis, we implicitly also proved another property of the SALM iterates. Andreani et al. define the CAKKT property (Complementary Approximate KKT) in \cite{Andreani_2010} as a necessary optimality condition for smooth problems, which can be readily extended to non-smooth problems in the following way.

\begin{definition}
$x^*\in \R^n$ fulfills \emph{CAM-stationarity} (Complementarity Approximate M-stationarity), if $x^*$ is feasible and there exists sequences $\{x^k\},\{\eta^k\}\subseteq \R^n$, $\{\lambda^k\}\subseteq \R^m_+$ and $\{\mu^k\}\subseteq \R^p$  with $\lambda_i^k\geq 0$ for all $i\in I(x^*)$ and $\lambda_i^k=0$ for all $i\notin I(x^*)$ such that the following conditions are fulfilled.
\begin{enumerate}[label = (\alph*)]
 \item $x^k\to x^*$.
 \item For all $k\in \N$ we have $\eta^k\in \nabla f(x^k)+\partial \varphi(x^k)+\sum_{i=1}^m \lambda^k_i\nabla g_i(x^k)+\sum_{j=1}^p \mu_j^k \nabla h_j(x^k)$.
 \item $\|\eta^k\|\to 0$.
 \item $\lim_{k\to\infty}\sum_{i\in I(x^*)} |\lambda_i^kg_i(x^k)|+\sum_{j=1}^p |\mu_j^k h_j(x^k)|=0$.
\end{enumerate}
\end{definition}

A proof that CAM-stationarity is in fact a necessary optimality condition
for our non-smooth program \eqref{optProb} is given in the appendix, 
Theorem \ref{thm:CAMnecessary}.
Combining the previous results from this section, we obtain the following corollary.

\begin{corollary}
 Let $x^*$ be feasible, fulfilling P\L{}CQ and there exists an infinite subset $K\subseteq \N$ such that the sequence $\{x^k\}$ generated by the SALM converges on $\K$ to $x^*$. Then $x^*$ is CAM-stationary.
\end{corollary}
\begin{proof}
 If we consider $\{x^k\}_{K}$, property (a) holds by assumption. By the SALM algorithm, we solve the subproblems as $\varepsilon_k$-stationary, which in turn implies the existence of a sequence $\{\eta^k\}\subseteq \R^n$ satisfying properties (b) and (c).
 Lastly, by Theorem~\ref{thm:lagrangeBounded}, the P\L{}CQ in $x^*$ guarantees that the correspoonding Lagrange multiplier subsequences generated by SALM stay bounded, which already implies (d). Also, by the reasoning from the complementarity section of the proof of Theorem \ref{thm:bounded->KKT}, we have $\lambda^k_i=0$ for $i\notin I(x^*)$ and $k\in K$ sufficiently large. Hence all requirements of the sequences $\{\lambda^k\}\subseteq \R^m_+$ and $\{\mu^k\}\subseteq \R^p$ are fulfilled by the SALM-multipliers.
\end{proof}

\section{On the Local Lipschitz Assumption of $\varphi$}
\label{sect:Lip}

For some applications of the SALM it would be desirable, not to demand the local Lipschitz property of $\varphi$ in Theorem \ref{thm:lagrangeBounded}. This issue arises for example when considering sparsity optimization: In order to obtain sparse solutions of a smooth optimization problem, it is popular to set $\varphi$ to a $\ell^0$-regularization-term and consider the composite problem. The $\ell^0$-pseudonorm is defined by first setting
\begin{align}\label{countingTerm}
|a|_0:=\begin{cases}
0&\text{ if }a=0, \\
1&\text{otherwise}
\end{cases}
\end{align}
for $a\in \R$ and then setting
\begin{align*}
\|x\|_0:=\sum_{i=1}^n|x_i|_0
\end{align*}
which is essentially counting the number of non-zero entries of the vector $x\in\R^n$. Note that despite the notation, this does not define a norm on $\R^n$. The function is also obviously not continuous in zero, and hence also not locally Lipschitz continuous there.
However, $\|\cdot\|_0$ is proper and lsc.

The following example shows that requiring $\varphi$ to be only proper and lsc does not necessarily lead to bounded Lagrange multipliers under the P\L{}CQ which would be required for the applicability of Theorem~\ref{thm:bounded->KKT}. The core observation is that without local Lipschitz continuity, we cannot rely on 
Lemma~\ref{prop:LipschitzSubdiffBounded} and may get arbitrarily large 
Mordukhovich subgradients, which in turn impact the Lagrange multipliers.

\subsection{Non-Lipschitz Example With Unbounded Multiplier Sequences}
\label{sect:counterex}

Let $f\equiv 0$ and $\varphi:\R\to\R$ be defined as (cf. Figure \ref{fig:counterex})
\begin{align*}
\varphi(x):=\begin{cases}
0 &\text{if }x=0\text{ or } x=\frac{1}{n}\text{ for }n\in\N, \\
1 &\text{else.}
\end{cases}
\end{align*}
\begin{figure}[h]
\centering
\begin{tikzpicture}
\begin{axis}[
    axis lines=middle,
    xmin=-0.05, xmax=1.05,
    ymin=-0.1, ymax=1.2,
    xtick={0,0.2,0.4,0.6,0.8,1},
    ytick={0,1},
]

\addplot[domain=0:1, thick] {1};

\addplot[only marks, mark=o] coordinates {(0,1)};
\addplot[only marks, mark=*] coordinates {(0,0)};

\foreach \n in {1,...,100} {
    \pgfmathsetmacro{\x}{1/\n}

    \addplot[only marks, mark=o] coordinates {(\x,1)};

    \addplot[only marks, mark=*] coordinates {(\x,0)};
}

\end{axis}
\end{tikzpicture}
\caption{Plot of $\varphi$}
\label{fig:counterex}
\end{figure}
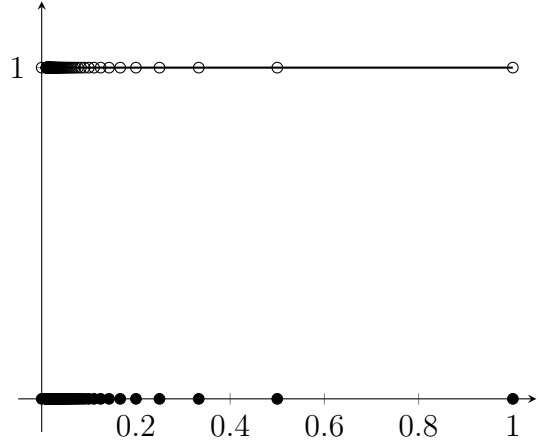
Then $\varphi$ is proper and lsc. We also impose two equality constraints with $h_1(x)=x$ and $h_2(x)=x^2$. The resulting optimization problem has the following properties.

\begin{lemma}\label{lemma:counterex}
Consider the optimization problem
\begin{equation}\label{eq:counterexProblem}
\min_{x\in\R}\varphi(x)\quad\text{s.t.}\quad x=0,\,x^2=0.
\end{equation}
Then the following statements hold.
\begin{enumerate}[label = (\alph*)]
\item For all $n\in\N$, we have $\partial \varphi\left(\frac{1}{n}\right)=\R$.
\item P\L{}CQ holds in $x^*=0$.
\end{enumerate}
\end{lemma}

\begin{proof}
(a) Let $n\in\N$. Then we have $\varphi\left(\frac{1}{n}\right)=0$ and
\begin{align}\label{eq:Umgebung1/n}
\varphi(x)=1\text{ for all }x\in \left(\frac{1}{n}-\frac{1}{n+1},\frac{1}{n}+\frac{1}{n+1}\right)\setminus\left\lbrace \frac{1}{n}\right\rbrace
\end{align}
Now let $a\in \R$. Let $\{x^k\}\subseteq \R$ be arbitrary with $x^k\to \overline{x}:=\frac{1}{n}$ and $ x^k \neq \overline{x} $. Then by \eqref{eq:Umgebung1/n}, we have for $k\in\N$ sufficiently large that
$\varphi(x^k)=1$ and thus
\begin{align*}
\frac{\varphi(x^k)-\varphi(\overline{x})-a(x^k-\overline{x})}{|x^k-\overline{x}|}\geq \frac{1-|a||x^k-\overline{x}|}{|x^k-\overline{x}|}=\frac{1}{|x^k-\overline{x}|}-|a|.
\end{align*}
We obtain
\begin{align*}
\liminf_{k\to\infty}\frac{\varphi(x^k)-\varphi(\overline{x})-a(x^k-\overline{x})}{|x^k-\overline{x}|}\geq \liminf_{k\to\infty}\frac{1}{|x^k-\overline{x}|}-|a|=+\infty\geq 0
\end{align*}
This shows
\begin{align*}
a\in \hat{\partial} \varphi\left(\frac{1}{n}\right)\subseteq \partial \varphi\left(\frac{1}{n}\right)
\end{align*}
and, as $a\in \R$ was chosen as arbitrary,
\begin{align*}
\partial \varphi\left(\frac{1}{n}\right)=\R
\end{align*}
which verifies statement (a).
\smallskip

\noindent
(b) We define $\delta :=1 $ and $\nu := \frac{1}{\sqrt{2}}$. For all $x\in \R$ with $|x|<\delta =1$ we have
\begin{align*}
\Phi(x)&=\frac{1}{2}\|h(x)\|^2 \\
&=\frac{1}{2}(x^2+x^4) \\
&\leq \frac{1}{2}(x^2+4x^4+4x^6) \\
&=\frac{1}{2}(x+2x^3)^2 \\
&=\frac{1}{2}\left(\begin{pmatrix}
x \\ x^2
\end{pmatrix}^T
\begin{pmatrix}
1 \\ 2x
\end{pmatrix}
\right)^2 \\
&=\frac{1}{2}(h(x)^T\nabla h(x))^2 \\
&=\nu^2(\Phi'(x))^2.
\end{align*}
This shows that P\L{}CQ holds in $x^*=0$.
\end{proof}

We now show that the analogon of Theorem~\ref{thm:lagrangeBounded} does not hold, i.e., there exists a sequence in $\R$ converging $\varphi$-attentively to the solution and only feasible point $x^*=0$ of \eqref{eq:counterexProblem} fulfilling P\L{}CQ, which has unbounded corresponding Lagrange multipliers $\{\mu^k\}$.

We define the sequence $\{x^k\}:=\frac{1}{k}$ and show that it is generated by SALM. We note that in our case the subdifferential of the augmented Lagrangian function at $x$ is given by
\begin{align*}
\partial \mathcal{L}_{\rho}(x,u,v)=\partial \varphi(x) + (v+\rho h(x))^T\nabla h(x).
\end{align*}
By Lemma~\ref{lemma:counterex} (a) we have $\partial\varphi(x^k)=\R$ for all $k\in\N$. Starting in $x^1=1$ we can in each step choose $z^k\in \partial \varphi(x^k)$ such that
\begin{align*}
z^k:=-(v_{k-1}+\rho_{k-1} h(x^k))^T\nabla h(x^k)
\end{align*}
and we get $0\in \partial \mathcal{L}_{\rho_{k-1}}(x^k,u_{k-1},v_{k-1})$ for all $k\in \N$, so $x^k$ is indeed a $\varepsilon_k$-stationary point
of the augmented Lagrangian function in each step and for each choice of $\varepsilon_k\searrow 0$.
Moreover, $x^k$ also converges $\varphi$-attentively to the feasible point $x^*=0$, as $\varphi(x^k)=0$ for all $k\in \N$.

We set the parameters $v^k=0$ for all $k\in\N$, $\rho_0=1$, $\tau = \frac{1}{2}$, $\gamma =2$.
In our specific case, we then have the multiplier update
\begin{align}\label{eq:counterexMultUpdate}
\mu^k=\rho_{k-1}\begin{pmatrix}
x^k \\ (x^k)^2
\end{pmatrix}
=\rho_{k-1}\begin{pmatrix}
\frac{1}{k} \\ \frac{1}{k^2}
\end{pmatrix} .
\end{align}
Additionally we have for $k\in\N$
\begin{align*}
\|h(x^k)\|^2=\frac{1}{k^2}+\frac{1}{k^4}
\end{align*}
and thus
\begin{align*}
\frac{\|h(x^{k+1})\|^2}{\|h(x^k)\|^2}=\frac{\frac{1}{(k+1)^2}+\frac{1}{(k+1)^4}}{\frac{1}{k^2}+\frac{1}{k^4}}=\frac{\frac{k^2}{(k+1)^2}+\frac{k^2}{(k+1)^4}}{1+\frac{1}{k^2}}\to 1
\end{align*}
for $k\to\infty$. This implies the existence of a $K\in \N$ with
\begin{align*}
\|h(x^{k+1})\|^2 > \frac{1}{4} \|h(x^k)\|^2
\end{align*}
for all $k\geq K$. The penalty update rule implies $\rho_{k}\geq 2^{k-K}$ for all $k\geq K$ and we obtain with \eqref{eq:counterexMultUpdate}
\begin{align*}
\|\mu^k\|=\rho_{k-1}\left\|
\begin{pmatrix}
\frac{1}{k} \\ \frac{1}{k^2}
\end{pmatrix}
\right\|\geq 2^{k-1-K}\frac{1}{k} \to \infty
\end{align*}
by taking $k$ sufficiently large. Thus, this example results in an unbounded multiplier sequence $\{\mu^k\}$.

\subsection{Locally Lipschitz Continuous Regularization Functions}

The above example illustrates that the convergence theory we established in this paper is not directly applicable to optimization problems with $\ell_0$-sparsity-regularization. However, there exist ways of circumventing this issue. One of them is to introduce surrogate sparsity functions with similar behavior as the $\ell_0$-pseudonorm, but with nicer analytical properties. A popular approach presented in \cite[Section 3.1.4]{CuiBook} is to approximate the expression in \eqref{countingTerm} by a function $\psi_\delta:\R\to [0,1]$, parameterized by a $\delta>0$, and then define the surrogate $\ell_0$-function by
\begin{align*}
\varphi(x):=\sum_{i=1}^n\psi_{\delta_i}(x_i).
\end{align*}
The approximation is done such that $\psi_\delta$ satisfies the following three conditions.
\begin{enumerate}[label = (R\arabic*)]
\item $\lim_{\delta \searrow 0}\psi_\delta(t)=|t|_0$ for any $t\in \R$ (pointwise).
\item For all $\delta>0$, $\psi_\delta$ is symmetric on $\R$ (i.e., $\psi_\delta(-t)=\psi_\delta(t)$ for all $t\geq 0$) as well as concave and non-decreasing on $[0,\infty)$.
\item $\psi_\delta(t)=1$ for all $t$ with $|t|\geq \delta$.
\end{enumerate}
In the following, we recall the definition of two popular surrogate functions
which satisfy the above three properties and which are, in addition,
Lipschitz continuous (even globally). They are visualized in 
Figure~\ref{fig:surrogateSparsity} and described in the following.
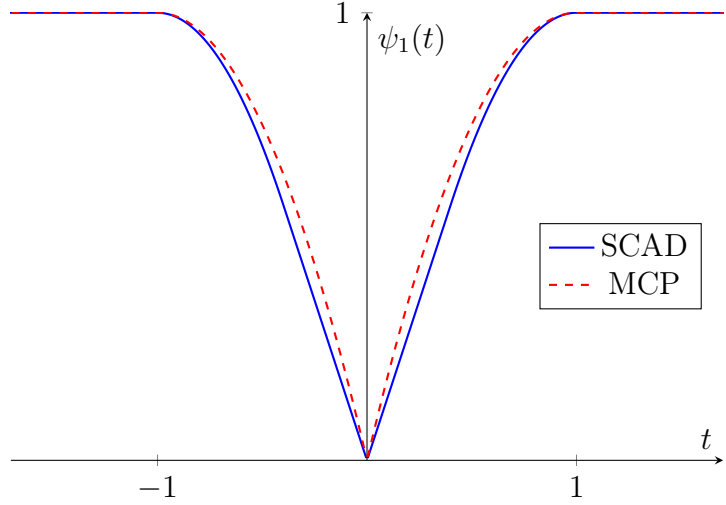
\begin{figure}
\centering
\begin{tikzpicture}
\begin{axis}[
    axis lines=middle,
    xlabel={$\,t$},
    ylabel={$\psi_{\delta}(t)$},
    xmin=-1.7, xmax=1.7,
    ymin=0, ymax=1,
    xtick distance =1,
    ytick distance =1,
    samples=400,
    domain=-2:2,
    width=11cm,
    height=7.5cm,
    legend style={at={(0.97,0.53)},anchor=north east}
]

\def\delta{1}
\def\a{2.5}      

\addplot[
    thick,
    blue
]
{
    abs(x) <= (\delta/\a) ?
        (2*\a*abs(x))/((\a +1)*\delta) :
    (abs(x) <= \delta ?
        1-((\delta-abs(x))^2)/((1-(1/(\a^2)))*\delta^2) :
        1
    )
};
\addlegendentry{SCAD}

\addplot[
    thick,
    dashed,
    red
]
{
    abs(x) <= \delta ?
        2*abs(x)/\delta - x^2/((\delta)^2) :
        1
};
\addlegendentry{MCP}

\end{axis}
\end{tikzpicture}
\caption{Two surrogate sparsity functions: MCP and SCAD (with $\delta =1$ and $a=2.5$)}
\label{fig:surrogateSparsity}
\end{figure}
\begin{example}\label{ex:MCP}
The \emph{minmax concave penalty (MCP)} family \cite{Zhang2010} is defined by
\begin{align*}
\psi^{\mathrm{MCP}}_\delta(t):=\begin{cases}
\frac{2|t|}{\delta}-\frac{t^2}{\delta^2}&\text{if }|t|\leq \delta, \\
1&\text{if }|t|> \delta.
\end{cases}
\end{align*}
\end{example}

Another possible choice, which was first introduced in \cite{Fan_2001}, is the SCAD family. The idea is similar to MCP, but the function is now split into three parts, namely a (piecewise) linear section in a neighborhood of zero, the constant part with value one for $|t|$ sufficiently large, and a quadratic part connecting the linear and constant sections while preserving continuous differentiability.
\begin{example}
The \emph{Smoothly Clipped Absolute Deviation (SCAD)} family parameterized by $a>2$ and $\delta > 0$ is defined by
\begin{align*}
\psi_{\delta,a}^{\mathrm{SCAD}}(t):=\begin{cases}
\frac{2a}{(a+1)\delta}|t| &\text{if }|t|\leq \frac{\delta}{a}, \\
1-\frac{(\delta-|t|)^2}{\left(1-\frac{1}{a^2}\right)\delta^2} &\text{if }\frac{\delta}{a}< |t| \leq \delta, \\
1 &\text{if }|t|\geq \delta.
\end{cases}
\end{align*}
\end{example}

We note that both functions are designed in a way such that they are continuously differentiable outside of zero, and even globally Lipschitz continuous on the
whole space $ \R $. Thus, both are a suitable choice for $\ell_0$-surrogates that fulfill Assumption \ref{assu:optProb} for our problem setting.

\section{Numerical Results}

\label{sec:numerical}

In this section we are testing the algorithm performance on a model problem and evaluate the relevant metrics. In order to create a flexible test environment tailored to our problem setting that is able to deal with a large range of problems, we implemented our own Python library.

\subsection{Implementation Details}

The implementation details of the SALM Algorithm~\ref{algo:SALM} are as follows: The sequences $\{u^k\}$ and $\{v^k\}$ are obtained in each step by projecting $\lambda^k$ and $\mu^k$ onto $[0,u_{\max}]^m$ and $[v_{\min},v_{\max}]^p$ respectively. The tolerance reduction for the sequence $\varepsilon_k$ is done in a similar way as in \emph{bazinga.jl}\footnote{Julia library for composite optimization, \url{https://github.com/aldma/Bazinga.jl}, \cite{De_Marchi_2023}\cite{demarchi2024implicit}}, namely setting a global tolerance $\mathrm{TOL}$ for termination checks, calculating an initial inner tolerance as $\sqrt[3]{\mathrm{TOL}}$ for the subproblems, and for each iteration in which the algorithm does not stop, multiplying the inner tolerance by a decrease factor $\kappa\in (0,1)$. As solvers of the subproblems, both the proximal gradient method and the quasi-Newton method PANOC\textsuperscript{+} \cite{Themelis_2018,De_Marchi_2022} are implemented. For the proximal gradient algorithm, the python library \texttt{copt}\footnote{\url{https://github.com/openopt/copt}, \cite{copt}} is used. The PANOC\textsuperscript{+} implementation was done by adapting the code from the julia Library\footnote{\url{https://github.com/JuliaFirstOrder/ProximalAlgorithms.jl}} by Alberto de Marchi into Python.

As a termination criterion for SALM, we check the following two conditions modeled after the M-stationarity conditions, we stop the algorithm if both conditions are fulfilled.
\begin{itemize}
\item The last subproblem has been solved with a tolerance of $\mathrm{TOL}_{\mathrm{dual}}$, indicating approximate stationarity of the augmented Lagrangian.
\item We have $\|h(x^k)\|_\infty<\mathrm{TOL}_{\mathrm{prim}}$ and $\|\min\{-g(x^k),\lambda^k\}\|_\infty<\mathrm{TOL}_{\mathrm{prim}}$, indicating approximate feasibility and complementarity.
\end{itemize}

For parameter values, we chose $u_{\max}=v_{\max}=-v_{\min}=10^8$, $\rho_0 =1$, $\kappa = 0.1$, $\gamma = 2$ and $\tau = 0.8$. We set the maximum iterations for SALM to $200$, and the ones for the subproblem solvers to $500$ per outer 
iteration $ k $. In all instances of SALM in this section, we chose a global tolerance of
\begin{align*}
\mathrm{TOL}=\mathrm{TOL}_{\mathrm{prim}}=\mathrm{TOL}_{\mathrm{dual}}=10^{-6}.
\end{align*}

\subsection{Sparse Portfolio Optimization}

The test problem from the field of mathematical finance is the sparse formulation of a portfolio optimization problem, specifically the Markowitz mean-variance model (cf.\ \cite{optimization_finance}). The problem in its original form is formulated as
\begin{align*}
\min_{x}\frac{1}{2}x^TQx \quad\text{s.t.}\quad \mu^Tx\geq \varrho,\quad \mathbf{1}_n^Tx =1 ,\quad x\geq 0.
\end{align*}
The vector $x\in\R^n$ is seen as the amount invested in a collection of $n$ financial assets. Expected risk $x^TQx$ and return $\mu^Tx$ are modeled by determining the covariance matrix $Q\in \R^{n\times n}$ and a vector $\mu$ of mean values of each asset, which have been calculated from past stock market performance. The value $\varrho>0$ indicates a minimal return threshold. We assume to have a budget of $1$. The equality constraint then indicates that the whole budget must be spent on assets. Finally, we want to invest money, hence only allowing positive spending.

Splitting the investment onto a broad range of assets with potentially small amounts per position may not be feasible in practice due to for example fixed transaction costs for each purchase \cite{Chen_2017}. Thus, we modify the original problem in order to obtain a sparse solution to the portfolio optimization:
\begin{align}\label{eq:portfolioSparse}
\min_{x}\frac{1}{2}x^TQx+\beta\varphi(x) \quad\text{s.t.}\quad \mu^Tx\geq \varrho,\quad \mathbf{1}_n^Tx =1,\quad x\geq 0,
\end{align}
with a sparsity promoting function $\varphi$ and regularization factor $\beta>0$, we obtain a composite formulation this way. A natural choice would be $\varphi = \|\cdot\|_0$. However, in view of the presented convergence theory, we chose the $\ell_0$-norm surrogate $\varphi = \varphi_\delta^{\mathrm{MCP}}$ from 
Example~\ref{ex:MCP} with $\delta=0.1$ for testing our SALM algorithm. Determining the proximal operator of the MCP function is an elementary, but quite arduous task, we refer to \cite{Shen_2019} for more details, where the prox-operator is determined for a slightly modified version of MCP.

We additionally ran the problem with $\ell_p^p$-regularization for $0<p<1$. This function is in particular not locally Lipschitz in zero, which we use to demonstrate the numerical differences in the algorithmic performance of SALM, especially in view of the Lipschitz properties discussed in Section \ref{sect:Lip}. The $\ell_p^p$-quasinorm admits analytical expressions for its proximal operator if $p=\frac{1}{2}$ or $p=\frac{2}{3}$ \cite{Liu_2024}. We chose $p=\frac{1}{2}$ for our numerical experiments.

We obtained the problem data from \cite{Frangioni_2007}\footnote{\url{https://commalab.di.unipi.it/datasets/mv/}}. It contains three sets $P_1,P_2,P_3$ consisting each of $10$ problems with randomly generated data and dimension $n=200$. Each set was generated by the authors according to a desired diagonal dominance ratio of the matrices $Q$, for more details we refer to the documentation.

On each of the $30$ problems we ran the SALM algorithm with the proximal gradient method (PG) and PANOC\textsuperscript{+} as subproblem solvers. As a reference solver, we ran the Gurobi solver\footnote{\url{https://www.gurobi.com}} on an equivalent MIQP\footnote{Mixed Integer Quadratic Programming} formulation of the $\ell_0$-regularized problem \eqref{eq:portfolioSparse} obtained from \cite{Feng2013}. For the regularization, we found a good working setup by choosing a common regularization factor for each of the three problem classes as indicated in Table \ref{tab:portfolioRegularization}. We compiled the results in two tables, namely Table \ref{tab:portfolio_MCP} for MCP regularization and Table \ref{tab:portfolioLp12} for $\ell_{p}^{p}$ regularization. We use the following abbreviations in reference to the terminal point $x^*$ of the algorithm:
\begin{itemize}
 \item Method: The SALM implementations are labeled by the subproblem solver used, PG standing for SALM with proximal gradient, PANOC+ for the SALM with PANOC\textsuperscript{+}.
 \item S: Status, $0$ when the termination criterion was met, $1$ if the algorithm failed to find an acceptable solution before the maximum number of iterations was reached.
 \item OI: Outer iterations of the SALM Loop.
 \item II: Cumulated inner iterations of the subproblem solver over all SALM steps.
 \item Obj.: Value of the objective function $\Psi(x^*)=\frac{1}{2}(x^*)^T Q x^* + \beta\varphi(x^*)$.
 \item NNZ: Number of non-zero entries in $x^*$.
 \item RE: relative Error calculated by
 \begin{align*}
\mathrm{RE}=\frac{\mathrm{obj}_{\mathrm{SALM}}-\mathrm{obj}_{\mathrm{Gurobi}}}{\mathrm{obj}_{\mathrm{Gurobi}}},
\end{align*}
\item Risk: expected risk term $\frac{1}{2}(x^*)^T Q x^*$.
\item Return: expected return $\mu^T x^*$.
\item $\lambda$: Multiplier for the minimum return constraint $\mu^T x\geq \varrho$.
\item $\mu$: Multiplier for the budget constraint $\mathbf{1}_n^Tx = 1$.
\end{itemize}

It is worth to note that we incorporated the bound constraint $x\geq 0$ directly into the formulations of the proximal operators.
As a starting vector we chose $x_0=\frac{1}{n}\cdot \mathbf{1}_n$.
\begin{table}[htbp]
\label{tab:portfolioRegularization}
\centering
\begin{tabular}{@{}llr@{}}
\toprule
Problem Class & Problem Class Name & $\beta$ \\ \midrule
$P_1$         & \texttt{orl200-005} & $25$    \\
$P_2$         & \texttt{orl200-05} & $100$   \\
$P_3$         & \texttt{pard200} & $50$    \\ \bottomrule
\end{tabular}
\caption{Regularization parameters for the portfolio problem classes}
\end{table}

With MCP regularization, SALM was able solve all problems except the instance of \texttt{pard200\_d} with the PG subproblem solver. For this particular problem, we also see that the multipliers grow quite large, probably indicating unbounded behavior, as in all other instances the multipliers are very tame, being zero in all except one other problem instance. We also see that all successfully obtained solutions have a sparsity of $2$ to $6$ non-zero entries, which are all quite comparable to the reference Gurobi values.

Another, albeit somewhat expected, effect we can observe is that the quasi-Newton approach of PANOC\textsuperscript{+} generally leads to convergence in fewer iterations in the subproblems compared to the proximal gradient method. This is especially visible when we look at the performance profile for the inner iterations in Figure \ref{fig:PPinnerIt}.

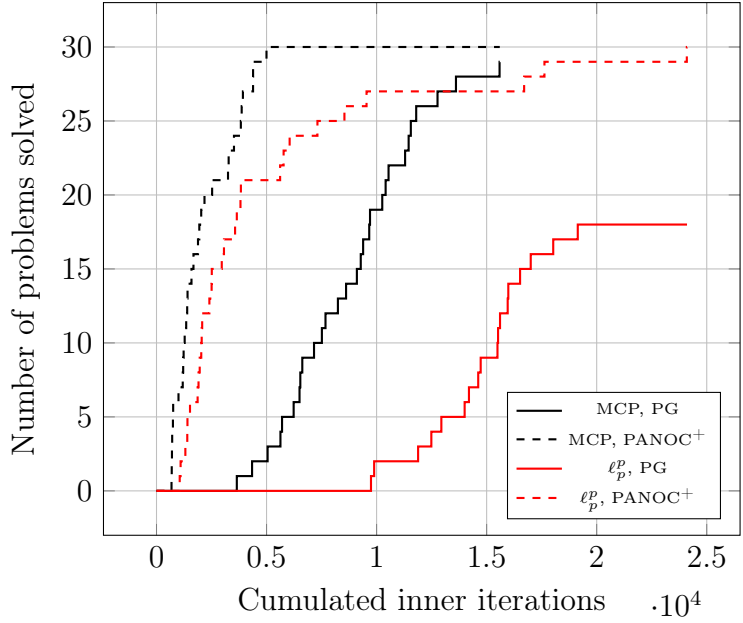
\begin{figure}
\centering
\begin{tikzpicture}
\begin{axis}[
    width = 10cm,
    xlabel={Cumulated inner iterations},
    ylabel={Number of problems solved},
    grid=major,
    legend pos=south east,
    ymajorgrids=true,
    xmajorgrids=true,
    legend style={font=\tiny}
]

\addplot[
    thick,
    const plot
] table[
    x=tau,
    y=rho,
    col sep=comma
] {ppPGinnerItMCP.csv};
\addlegendentry{MCP, PG}

\addplot[
    thick,
    dashed,
    const plot
] table[
    x=tau,
    y=rho,
    col sep=comma
] {ppPANOCinnerItMCP.csv};
\addlegendentry{MCP, PANOC\textsuperscript{+}}

\addplot[
    red,
    thick,
    const plot
] table[
    x=tau,
    y=rho,
    col sep=comma
] {ppPGinnerItLp.csv};
\addlegendentry{$\ell_{p}^{p}$, PG}

\addplot[
    red,
    thick,
    dashed,
    const plot
] table[
    x=tau,
    y=rho,
    col sep=comma
] {ppPANOCinnerItLp.csv};
\addlegendentry{$\ell_{p}^{p}$, PANOC\textsuperscript{+}}

\end{axis}
\end{tikzpicture}
\caption{Performance profile of cumulated inner iterations}
\label{fig:PPinnerIt}
\end{figure}

However, when looking at computation time, we see that PG performs slightly better, indicating that the PANOC\textsuperscript{+}-iterations require a significantly higher computational effort than the ones for proximal gradient, cf. Figure \ref{fig:PPCPU}.

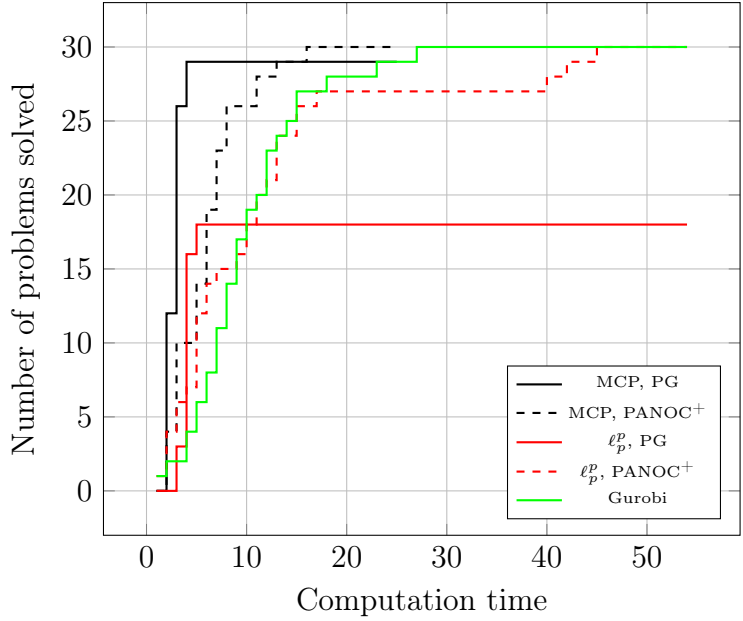
\begin{figure}
\centering
\begin{tikzpicture}
\begin{axis}[
    width = 10cm,
    xlabel={Computation time},
    ylabel={Number of problems solved},
    grid=major,
    legend pos=south east,
    ymajorgrids=true,
    xmajorgrids=true,
    legend style={font=\tiny}
]

\addplot[
    thick,
    const plot
] table[
    x=tau,
    y=rho,
    col sep=comma
] {ppPGcpuTimeMCP.csv};
\addlegendentry{MCP, PG}

\addplot[
    thick,
    dashed,
    const plot
] table[
    x=tau,
    y=rho,
    col sep=comma
] {ppPANOCcpuTimeMCP.csv};
\addlegendentry{MCP, PANOC\textsuperscript{+}}

\addplot[
    red,
    thick,
    const plot
] table[
    x=tau,
    y=rho,
    col sep=comma
] {ppPGcpuTimeLp.csv};
\addlegendentry{$\ell_{p}^{p}$, PG}

\addplot[
    red,
    thick,
    dashed,
    const plot
] table[
    x=tau,
    y=rho,
    col sep=comma
] {ppPANOCcpuTimeLp.csv};
\addlegendentry{$\ell_{p}^{p}$, PANOC\textsuperscript{+}}

\addplot[
    green,
    thick,
    const plot
] table[
    x=tau,
    y=rho,
    col sep=comma
] {ppGurobicpuTime.csv};
\addlegendentry{Gurobi}

\end{axis}
\end{tikzpicture}
\caption{Performance profile of for computation time in seconds}
\label{fig:PPCPU}
\end{figure}

The same behavior can also be observed when comparing the subproblem solvers for the $\ell_{p}^{p}$-regularized problems also displayed in the same performance profiles. But in contrast to MCP, we here have a total number of $12$ problem instances (all using the proximal gradient subproblem solver), in which the algorithm fails to find an acceptable solutions in time. Otherwise, the comparability of the results to the ones from Gurobi and MCP is limited, as the latter ones always contribute a constant value outside (a neighborhood) of zero, whereas the $\ell_{p}^{p}$-regularization scales infinitely with the distance from zero. This property also leads to a significantly lower sparsity in the result for this regularization method. We also see that especially in the case of failure, but also for successful instances, the multipliers take on higher values than the ones for MCP regularization.

We interpret these results as evidence for our theoretical necessity of local Lipschitzness of $\varphi$ in Theorems \ref{thm:bounded->KKT}, \ref{thm:lagrangeBounded}, Section \ref{sect:counterex} in order to ensure sensible behavior of the multiplier sequences, and consequently better convergence guarantees of SALM.

\section{Conclusions}\label{Sec:Conclusions}

We presented the Safeguarded Augmented Lagrangian Method as an algorithm capable of performing optimization on composite constrained problems with a locally Lipschitz continuous regularization function. We obtained a global convergence result to an M-stationary point under the minimally restrictive P\L{}CQ, equivalently under the error bound condition. The boundedness of the Lagrange multipliers played a central role in the convergence analysis, and not imposing a Lipschitz property onto $\varphi$ may lead to unbounded multipliers, as demonstrated in the counterexample. Our numerical findings underline our theoretical considerations about the choice of a regularization function.

Recently, De Marchi et al.\ introduced the Elastic SALM in \cite{elasticSALM}, a modification of the SALM algorithm that allows the safeguards to grow in a controlled way. This approach has the advantage of unifying the good global convergence guarantees of SALM with some desirable properties of the classical ALM when applied to convex problems.
As to the best of our knowledge, it is not yet clear, if Elastic SALM also has a global convergence property for potentially non-convex problems under EB/P\L{}CQ, which might be a topic for further research.

\printbibliography

\appendix

\section{Appendix}

This proof for CAM-stationarity as a necessary optimality condition is a direct extension of \cite[Theorem 3.3]{Andreani_2010} to the composite setting.

\begin{theorem}\label{thm:CAMnecessary}
 Let $x^*$ be a local minimizer of \eqref{optProb}. Then $x^*$ is CAM-stationary.
\end{theorem}

\begin{proof}
We denote $\Psi:= f+\varphi$. As $x^*$ is a local minimizer, there exists a $\delta>0$ such that $\Psi(x^*)\leq \Psi(x)$ for all
feasible $x\in \overline{B}_\delta (x^*)$. We consider the auxiliary problem
\begin{align}\label{auxProb1}
 \min_{x\in \R^n}\Psi(x) +\frac{1}{2}\|x-x^*\|^2 \quad \text{s.t.}\quad g(x)\leq 0, \quad  h(x)=0, \quad x\in \overline{B}_\delta(x^*).
 \end{align}
By construction, $x^*$ is the unique solution of Problem \eqref{auxProb1}. For each $k\in \N$, we define another auxiliary problem as follows:
\begin{align}\label{auxProb2}
 \min_{x\in \R^n}\Psi(x) +\frac{1}{2}\|x-x^*\|^2+\frac{k}{2}\left[\|g(x)_+\|^2+\|h(x)\|^2\right] \quad \text{s.t.} \quad x\in \overline{B}_\delta(x^*).
 \end{align}
As we have a continuous objective function on a compact set, there exists a solution $x^k$ to \eqref{auxProb2} for each $k\in \N$.
By the feasibility of $x^*$ and the optimality of each $x^k$ with respect to \eqref{auxProb2}, we have the property
\begin{align}\label{app:xk-optimality}
 \Psi(x^k)+\frac{1}{2}\|x^k-x^*\|^2+\frac{k}{2}\left[\|g(x^k)_+\|^2+\|h(x^k)\|^2\right]\leq \Psi(x^*).
\end{align}
 Furthermore, as the sequence $\{x^k\}\subseteq \overline{B}_\delta(x^*)$ is bounded, there exists an infinite set $ K\subseteq \N$ and $z^*\in \overline{B}_\delta(x^*)$ such that
 \begin{align*}
  \lim_{k\to\infty,k\in K}x^k = z^*.
 \end{align*}
 We can now divide expression \eqref{app:xk-optimality} by $k$ and take the limit $k\to\infty$ for $k\in K$. This yields
 \begin{align*}
  \|g(z^*)_+\|^2+\|h(z^*)\|^2=0,
 \end{align*}
 showing feasibility of $z^*$ for problem \eqref{auxProb1}. Additionally, we have
 \begin{align*}
   \Psi(x^k)+\frac{1}{2}\|x^k-x^*\|^2\leq \Psi(x^k)+\frac{1}{2}\|x^k-x^*\|^2+\frac{k}{2}\left[\|g(x^k)_+\|^2+\|h(x^k)\|^2\right]\leq \Psi(x^*),
 \end{align*}
and taking the limit $k\to\infty$, $k\in K$ yields
\begin{align*}
 \Psi(z^*)+\frac{1}{2}\|z^*-x^*\|\leq \Psi(x^*),
\end{align*}
which implies $z^* = x^*$ by the optimality of $x^*$ for Problem \eqref{auxProb1}, thus showing
\begin{align}\label{app:xk-convergence}
 \lim_{k\to\infty,k\in K}x^k= x^*.
\end{align}

We now define the sequences $\lambda^k:=k g(x^k)_+$, $\mu^k:=k h(x^k)$ and $\eta^k:=x^*-x^k$ for $k\in \K$. Clearly, $\lambda^k \geq 0$ and for $i\notin I(x^*)$, by the continuity of $g$ we have $g_i(x^k)\leq 0$ and hence $\lambda_i^k=0$ for $k$ sufficiently large. We take $ K_1\subseteq K$ such that $\lambda_i^k=0$ for all $i\notin I(x^*)$ and $k\in K_1$ and show
that all properties of CAM-stationarity hold for the above defined sequences on the index set $ K_1$.

In \eqref{app:xk-convergence} we showed property (a) of CAM-stationarity for the sequence $\{x_k\}_{k\in K_1}$. By definition, $\{\eta^k\}_{k\in K_1}$ fulfills (c).

By local optimality with respect to the auxiliary problem \eqref{auxProb2}, the generalized Fermat rule from Lemma \ref{lem:genFermat} and the Mordukhovich subdifferential sum rule yields
\begin{align*}
 0\in \nabla f(x^k)+\partial \varphi(x^k)+x^k-x^*+\sum_{i=1}^m kg_i(x^k)_+\nabla g_i(x^k)+\sum_{j=1}^p k h_j(x^k)\nabla h_j(x^k) ,
\end{align*}
which is equivalent to
\begin{align*}
 \eta^k = x^*-x^k\in \nabla f(x^k)+\partial \varphi(x^k)+\sum_{i=1}^m \lambda_i^k\nabla g_i(x^k)+\sum_{j=1}^p \mu_j^k\nabla h_j(x^k),
\end{align*}
which shows property (b), whereas (c) holds by definition of $ \eta_k $
and (a).

Once again considering \eqref{app:xk-optimality} and taking the limit $k\to \infty$ for $k\in K_1$ yields
\begin{align*}
 \lim_{k\to\infty,k\in K_1}\sum_{i=1}^m k (g_i(x^k)_+)^2 +\sum_{j=1}^p k h_j(x^k)^2 =0.
\end{align*}
As
\begin{align*}
 |\mu^k_jh_j(x^k)| = |kh_j(x^k)^2| = kh_j(x^k)^2
\end{align*}
and
\begin{align*}
 |\lambda_i^k g_i(x^k)| = |kg_i(x^k)_+g_i(x^k)| = |k(g_i(x^k)_+)^2| = k(g_i(x^k)_+)^2,
\end{align*}
as well as $\lambda_i^k=0$ for all $k\in K_1$ and $i\notin I(x^*)$, we have
\begin{align*}
 \lim_{k\to\infty,k\in K_1}\sum_{i\in I(x^*)} |\lambda_i^k g_i(x^k)| +\sum_{j=1}^p |\mu^k_jh_j(x^k)| =0,
\end{align*}
which shows the last property (d).
\end{proof}

\pgfplotstableread[col sep=comma, empty cells with={--}]{portfolio_MCP.csv}\loadedtable

\begin{table}[htbp]
\tiny
\centering

\pgfplotstabletypeset[
    col sep=comma,
    columns={problemName,methodName,status,outerIterations,innerIterations,obj,nnz,relativeError,risk,ReturnVal,lambda,mu},
    columns/problemName/.style={
    column name=Problem,
    string type,
    postproc cell content/.append code={
        \pgfmathtruncatemacro{\rowmod}{mod(\pgfplotstablerow,3)}%
        \ifnum\rowmod=0
        \else
            \pgfkeysalso{@cell content={}}%
        \fi
    }
},
    columns/methodName/.style={column name=Method, string type, column type = l},
    columns/status/.style={column name=S,
    		column type=r,
    		fixed,
    		postproc cell content/.append code={
        		\pgfmathtruncatemacro{\rowmod}{mod(\pgfplotstablerow,3)}%
        		\ifnum\rowmod=0
            		\pgfkeysalso{@cell content={--}}%
        		\fi
    		}
	},
    columns/outerIterations/.style={column name=OI,
    		column type=r,
    		fixed,
    		postproc cell content/.append code={
        		\pgfmathtruncatemacro{\rowmod}{mod(\pgfplotstablerow,3)}%
        		\ifnum\rowmod=0
            		\pgfkeysalso{@cell content={--}}%
        		\fi
    		}
	},
    columns/innerIterations/.style={column name=II,
    		column type=r,
    		fixed,
    		postproc cell content/.append code={
        		\pgfmathtruncatemacro{\rowmod}{mod(\pgfplotstablerow,3)}%
        		\ifnum\rowmod=0
            		\pgfkeysalso{@cell content={--}}%
        		\fi
    		}
	},
    columns/obj/.style={column name=Obj., column type=r,fixed,precision=3},
    columns/nnz/.style={column name=NNZ, column type=r},
    columns/relativeError/.style={column name=RE,
    		column type=r,
    		fixed,
    		precision=3,
    		postproc cell content/.append code={
        		\pgfmathtruncatemacro{\rowmod}{mod(\pgfplotstablerow,3)}%
        		\ifnum\rowmod=0
            		\pgfkeysalso{@cell content={--}}%
        		\fi
    		}
	},
    columns/risk/.style={column name=Risk, column type=r,fixed,precision=3},
    columns/ReturnVal/.style={column name=Return, column type=r,fixed,precision=5},
    columns/lambda/.style={
    column name=$\lambda$,
    column type=r,
    sci,
    precision=2,
    postproc cell content/.append code={
        \pgfmathtruncatemacro{\rowmod}{mod(\pgfplotstablerow,3)}%
        \ifnum\rowmod=0
            \pgfkeysalso{@cell content={--}}%
        \fi
    }
},
columns/mu/.style={
    column name=$\mu$,
    column type=r,
    sci,
    precision=2,
    postproc cell content/.append code={
        \pgfmathtruncatemacro{\rowmod}{mod(\pgfplotstablerow,3)}%
        \ifnum\rowmod=0
            \pgfkeysalso{@cell content={--}}%
        \fi
    }
},
    every head row/.style={before row=\toprule, after row=\midrule},
    every last row/.style={after row=\bottomrule}
]{\loadedtable}

\caption{Sparse portfolio optimization results with MCP-regularization}
\label{tab:portfolio_MCP}

\end{table}

\pgfplotstableread[col sep=comma, empty cells with={--}]{portfolio_Lp12.csv}\loadedtable

\begin{table}[htbp]
\tiny
\centering

\pgfplotstabletypeset[
    col sep=comma,
    columns={problemName,methodName,status,outerIterations,innerIterations,obj,nnz,relativeError,risk,ReturnVal,lambda,mu},
    columns/problemName/.style={
    column name=Problem,
    string type,
    postproc cell content/.append code={
        \pgfmathtruncatemacro{\rowmod}{mod(\pgfplotstablerow,3)}%
        \ifnum\rowmod=0
        \else
            \pgfkeysalso{@cell content={}}%
        \fi
    }
},
    columns/methodName/.style={column name=Method, string type, column type = l},
    columns/status/.style={column name=S,
    		column type=r,
    		fixed,
    		postproc cell content/.append code={
        		\pgfmathtruncatemacro{\rowmod}{mod(\pgfplotstablerow,3)}%
        		\ifnum\rowmod=0
            		\pgfkeysalso{@cell content={--}}%
        		\fi
    		}
	},
    columns/outerIterations/.style={column name=OI,
    		column type=r,
    		fixed,
    		postproc cell content/.append code={
        		\pgfmathtruncatemacro{\rowmod}{mod(\pgfplotstablerow,3)}%
        		\ifnum\rowmod=0
            		\pgfkeysalso{@cell content={--}}%
        		\fi
    		}
	},
    columns/innerIterations/.style={column name=II,
    		column type=r,
    		fixed,
    		postproc cell content/.append code={
        		\pgfmathtruncatemacro{\rowmod}{mod(\pgfplotstablerow,3)}%
        		\ifnum\rowmod=0
            		\pgfkeysalso{@cell content={--}}%
        		\fi
    		}
	},
    columns/obj/.style={column name=Objective, column type=r,fixed,precision=3},
    columns/nnz/.style={column name=NNZ, column type=r},
    columns/relativeError/.style={column name=RE,
    		column type=r,
    		fixed,
    		precision=3,
    		postproc cell content/.append code={
        		\pgfmathtruncatemacro{\rowmod}{mod(\pgfplotstablerow,3)}%
        		\ifnum\rowmod=0
            		\pgfkeysalso{@cell content={--}}%
        		\fi
    		}
	},
    columns/risk/.style={column name=Risk, column type=r,fixed,precision=3},
    columns/ReturnVal/.style={column name=Return, column type=r,fixed,precision=5},
    columns/lambda/.style={
    column name=$\lambda$,
    column type=r,
    sci,
    precision=2,
    postproc cell content/.append code={
        \pgfmathtruncatemacro{\rowmod}{mod(\pgfplotstablerow,3)}%
        \ifnum\rowmod=0
            \pgfkeysalso{@cell content={--}}%
        \fi
    }
},
columns/mu/.style={
    column name=$\mu$,
    column type=r,
    sci,
    precision=2,
    postproc cell content/.append code={
        \pgfmathtruncatemacro{\rowmod}{mod(\pgfplotstablerow,3)}%
        \ifnum\rowmod=0
            \pgfkeysalso{@cell content={--}}%
        \fi
    }
},
    every head row/.style={before row=\toprule, after row=\midrule},
    every last row/.style={after row=\bottomrule}
]{\loadedtable}

\caption{Sparse portfolio optimization results with $\ell^{\frac{1}{2}}$-regularization}
\label{tab:portfolioLp12}

\end{table}

\end{document}